\documentclass{amsart}
\usepackage[english]{babel}
\usepackage{latexsym}
\usepackage{amsmath,amssymb,amsthm}
\usepackage{mathrsfs}
\usepackage{epsfig}
\usepackage{caption}
\usepackage[all]{xy}
\usepackage{orcidlink}

\subjclass[2020]{14J60, 14H60}
\keywords{Ulrich´s bundles, syzygies bundles, projective varieties}

\newtheorem{Theorem}{Theorem}[section]
\newtheorem{Lemma}[Theorem]{Lemma}
\newtheorem{Definition}[Theorem]{Definition}
\newtheorem{Proposition}[Theorem]{Proposition}
\newtheorem{Corollary}[Theorem]{Corollary}

\newtheorem{Remark}[Theorem]{Remark}
\newtheorem{Example}[Theorem]{Example}

\title{On the Ulrichness of twisted syzygies and dual syzygies bundles} 

\author{ Hugo Torres-L\'opez and Alexis G. Zamora}

\address{Secihti-U. A. Matem\'aticas, U. Aut\'onoma de
Zacatecas
\newline  Calzada Solidaridad entronque Paseo a la
Bufa, \newline C.P. 98000, Zacatecas, Zac. M\'exico.}
\email{hugo@cimat.mx}

\address{U. A. Matem\'aticas, U. Aut\'onoma de
Zacatecas
\newline  Calzada Solidaridad entronque Paseo a la
Bufa, \newline C.P. 98000, Zacatecas, Zac. M\'exico.}
\email{alexiszamora@uaz.edu.mx}

\begin{document}

\maketitle

\begin{abstract}Given a projective variety $X$ and a very ample line bundle $\mathcal{L}$ on $X$, we classify for which $X$ and $\mathcal{L}$ the twisted syzygies and twisted dual syzygies bundles are Ulrich with respect to the polarizations $\mathcal{L}^a$. We obtain some partial results when considering an arbitrary polarization $H$.\end{abstract}

  \section{Introduction}

 Our starting point is a complex projective variety $X\subseteq \mathbb{P}^n$ that we always assume non-degenerate, non-singular, and irreducible. We simply say that $X\subseteq \mathbb{P}^n$ is a projective variety.

A cursive letter, like $\mathcal{L}$, will denote an invertible sheaf on a projective variety $X$ and $L$ will be its associated divisor, that is: $\mathcal{L}=\mathcal{O}_X(L)$.

We write $H^i(\mathcal{L})$ or $H^i(L)$ instead of $H^i(X,\mathcal{L})$ and analogously for the dimension $h^i(\mathcal{L})$.

If $\mathcal{L}:= i^* \mathcal{O}_{\mathbb{P}^n}(1)$, then $\mathcal{L}$ is very ample and $\mathbb{P}^n=\mathbb{P}(V)$ for some subspace $V\subseteq H^0(\mathcal{L})$, that we call an embedding subspace. Thus, if $\dim X=d$, then $\deg X=L^d$.

We have an exact sequence:

\begin{equation}\label{definition of M}
    0\to M_{\mathcal{L},V}\to V\otimes \mathcal{O}_X\to \mathcal{L}\to 0,\end{equation}
and its dual:
\begin{equation}\label{definition of M dual} 0\to \mathcal{L}^{-1} \to V\otimes \mathcal{O}_X\to M^{\vee}_{\mathcal{L},V}\to 0.\end{equation}

The locally free sheaf of rank $n$ $M_{\mathcal{L},V}$ is called the syzygy bundle of $\mathcal{L}$ with respect to $V$.

Recall from \cite{Beauville1} the following:

\begin{Definition}\label{definition of Ulrich} A locally free sheaf $\mathcal{E}$ on a $d-$dimensional projective variety $X$ is an $H-$Ulrich bundle if

$$H^i(\mathcal{E}(-pH))=0, \quad \text{for all } i \text{ and } 1\le p\le d. $$
\end{Definition}

In \cite{BUlrich}, B.
Ulrich initiated the study of Ulrich bundles in the context of maximally generated Cohen–Macaulay modules. He was studying conditions under which a local Cohen–Macaulay ring is Gorenstein. Through Beilinson monads and relative resolutions, Eisenbud and  Schreyer, showed that certain ACM sheaves admit completely linear resolutions (see \cite{Eisenbud}). This viewpoint established the bridge between geometry (coherent sheaves on projective varieties) and algebra (maximal Cohen–Macaulay modules with linear resolution).

Other relevant papers on the geometric properties of Ulrich bundles are \cite{BeauvilleHypersurface}, \cite{HartshorneCasanellas} and \cite{Coskun et al}.

There are beautiful introductions to Ulrich bundles, such as \cite{Beauville1} and \cite{Coskun Emre}, and even a textbook \cite{CostaMirog}.

The main goal of this article is to prove the following:

 \begin{Theorem}\label{Theorem Main} Let $X$ be a projective variety, $\mathcal{L}$ a very ample line bundle and $V\subseteq H^0(\mathcal{L})$ be an embedding subspace, $k\in \mathbb{Z}$, and $a\in \mathbb{N}$.
 
 \begin{itemize} 
 
 \item[i)] $M^{\vee}_{\mathcal{L}, V}\otimes \mathcal{L}^{k+1}$ is $\mathcal{L}^a-$Ulrich if and only if, either:
 
 \begin{itemize}
     \item[(1)]  $X=\mathbb{P}^1$, $L=\mathcal{O}_{\mathbb{P}^1}(2)$, $V=H^0(\mathcal{O}_{\mathbb{P}^1}(2))$ and $a=k+2$.
           \item[(2)] $X=\mathbb{P}^1$, $L=\mathcal{O}_{\mathbb{P}^1}(1)$, $V=H^0(\mathcal{O}_{\mathbb{P}^1}(1))$ and $a=k+3.$
    \item[(3)]$k=0$, $a=2$ and $X=\mathbb{P}^2$, that is for the sheaf $T\mathbb{P}^2$ with respect to $\mathcal{O}_{\mathbb{P}^2}(2)$.
 \end{itemize}
 \item[ii)] $M_{\mathcal{L}, V}^{k-1}\otimes \mathcal{L}$ is $\mathcal{L}^a-$Ulrich if and only if $a=k-1$ and  $X$ is a normal rational curve.
 \end{itemize}
\end{Theorem}

Recently, several authors have studied the Ulrichness of bundles similar to those studied in this paper. In \cite{Montero}, the authors proved that the only projective manifolds whose tangent bundle is Ulrich are the twisted cubic and the Veronese surface. Moreover, they prove that the cotangent bundle is never Ulrich. Later, in \cite{AngeloTwistingTangent}, the authors classify smooth varieties $X\subset \mathbb{P}^n$ such that the twisting tangent bundle $TX(k)$ is an Ulrich bundle for some $k\in \mathbb{Z}.$ In \cite{Angelo} and \cite{Lopez-Collino}, the author characterizes smooth irreducible varieties with Ulrich twisted normal or conormal bundles. These papers usually studied Ulrichness with respect to a hyperplane section of a given embedding. For instance, in \cite{Lopez-Collino}, Theorem \ref{Theorem Main} was proved in the case $a=1$.

In section \ref{Section some properties} we collect some properties of syzygies and Ulrich bundles that will be used in the sequel. In sections \ref{Item i)} and \ref{Item ii)} we prove, respectively, items i) and ii) of $\ref{Theorem Main}$. Finally, section \ref{section arbitrary H} gives some partial results on the Ulrichness of syzygy bundles with respect to arbitrary polarizations.

\section{Some properties of syzygies and Ulrich bundles}\label{Section some properties}
In the projective space $\mathbb{P}^n$ we have the Euler exact sequence:

$$ 0\to \mathcal{O}_{\mathbb{P}^n} \to \mathcal{O}_{\mathbb{P}^n}(1)^{\oplus n+1} \to T\mathbb{P}^n \to 0$$
with $T\mathbb{P}^n$ the tangent sheaf (see, for instance, \cite{HartshorneAlgebraicGeometry}, Theorem II.8.13). Twisting with $\mathcal{O}_X(k)$ and restricting to $X\subseteq \mathbb{P}^n$ embedded by a subspace $V\subseteq H^0(\mathcal{L})$, we obtain in combination with the exact sequence (\ref{definition of M}), the isomorphism:
\begin{equation}\label{tangent isomorphisms}
M^{\vee}_{\mathcal{L},V}\otimes \mathcal{L}^{k+1}= M^{\vee}_{\mathcal{O}_X(1), \mathbb{P}^n}\otimes \mathcal{O}_X(k+1)\simeq T\mathbb{P}^n(k)\vert_X. \end{equation}

Taking duals, we obtain:

\begin{equation}\label{cotangent ismorphisms}
M_{\mathcal{L},V}\otimes \mathcal{L}^{k-1}= M_{\mathcal{O}_X(1), \mathbb{P}^n}\otimes \mathcal{O}_X(k-1)\simeq \Omega\mathbb{P}^n(k)\vert_X. \end{equation}

Ulrich´s bundles admit many characterizations and enjoy several interesting geometric properties. Here we list some that will be used in the sequel.

\begin{Proposition}\label{proposition properties of Ulrich}
    Let $\mathcal{E}$ be an $H-$Ulrich bundle on an $d-$dimensional projective variety $X\subseteq \mathbb{P}^n$. Assume $rank(\mathcal{E})=r$ and $\deg X=m$. Then:
    \begin{itemize}
        \item[i)] $H^i(\mathcal{E}(-j))=0$ for all $j\in \mathbb{Z}$, $1\le j \le d-1$ and $h^0(\mathcal{E})=rm$.
        \item[ii)] if $\mathcal{E}$ is $H-$Ulrich, then $\mathcal{E}\vert_Y$ is $H\vert_Y-$Ulrich for $Y=X\cap H$.
        \item[iii)] $\mathcal{E}$ is $H-$semistable.
    \end{itemize}
\end{Proposition}

Moreover,  from the Definition \ref{definition of Ulrich} it immediately follows that:

\begin{Proposition}\label{PropositionUlrichExacteness} If 

$$0\to \mathcal{F}_1 \to \mathcal{E} \to \mathcal{F}_2\to 0$$
is an exact sequence on vector bundles on $X$ and $\mathcal{F}_i$, $i=1,2$ are Ulrich´s, then $\mathcal{E}$ also is.\end{Proposition}

The proof of Proposition \ref{proposition properties of Ulrich} can be found in \cite{Beauville1} or \cite{CostaMirog}.

\section{Proof of Theorem \ref{Theorem Main}, item i)}\label{Item i)}

We start by establishing the following:

\begin{Lemma}\label{Lemma for dual M}  Let $X$ be a projective variety and $L$,  $V\subseteq H^0(\mathcal L)$ an embedding subspace. If $M_{\mathcal{L},V}^{\vee}\otimes \mathcal{L}^{k+1}$ is $\mathcal{L}^a-$Ulrich for some $a\in \mathbb{N}$, then 

$$0\le k+1 <a;$$
unless $n=1$, $k=-2$ and $a=1$.
In particular, except for that case:
$$\vert (k+1-a)L\vert =\emptyset .$$ 

\end{Lemma}

\begin{proof}
Assume we are not in the exceptional case $n=a=1$ and $k=-2$. We claim that $h^0(M^{\vee}_{\mathcal{L}, V}\otimes \mathcal{L}^{k+1})\neq 0$ if and only if $k+1\geq 0.$ Note that  $h^0(M^{\vee}_{\mathcal{L}, V})\neq 0$ because $M^{\vee}_{\mathcal{L}, V}$ is globally generated. Therefore, for any $k+1\geq 0,$ we have $h^0(M^{\vee}_{\mathcal{L}, V}\otimes \mathcal{L}^{k+1})\neq 0.$
     
If $k+1<0$, we split the proof for the case $\dim X=1$ and $\dim X\ge 2$. If $X$ is a curve,  then the degree: $$\deg M_{\mathcal{L},V}^{\vee}\otimes \mathcal{L}^{k+1}= \deg \mathcal{L}+n(k+1)\deg \mathcal{L}= (n(k+1)+1)\deg \mathcal{L}$$
is negative. Since $M_{\mathcal{L},V}^{\vee}\otimes \mathcal{L}^{k+1}$ is semistable (Proposition \ref{proposition properties of Ulrich} iii)) and has negative degree, it follows that $h^0(M_{\mathcal{L},V}^{\vee}\otimes \mathcal{L}^{k+1})=0.$ 

If $\dim X\ge 2$, and $k\leq -2,$ by the Kodaira Vanishing Theorem (see \cite{GriffithsHarris}, Chapter 1.2) and from the exact sequence:

\begin{equation}\label{DualMkvariedades}
    0\to \mathcal{L}^{k}\to V\otimes \mathcal{L}^{k+1}\to M_{\mathcal{L},V}^{\vee}\otimes \mathcal{L}^{k+1}\to 0,
\end{equation}
we have $h^0(M_{\mathcal{L},V}^{\vee}\otimes \mathcal{L}^{k+1})=0$.

If $M_{\mathcal{L},V}^{\vee}\otimes \mathcal{L}^{k+1}$ is $\mathcal{L}^a$-Ulrich then $h^0(M^\vee_{\mathcal{L},V}\otimes \mathcal{L}^{k+1-a})=0$, and the Lemma follows from the claim.
 \end{proof}

 \begin{Remark}
     Note that, being $X\to \mathbb{P}(V)$ an embedding, $n=1$ simply means $X=\mathbb{P}^1$. In the exceptional case $n=1$, $k=-2$, the degree of $M^{\vee}_{\mathcal{L}, V}\otimes \mathcal{L}^{k+1}$ is $0$ and we don´t have the vanishing of the cohomology. In fact, as explained in the next proposition, this case gives rise to an Ulrich bundle.
 \end{Remark}

We prove Theorem \ref{Theorem Main} i) successively for curves, surfaces, and varieties of dimension at least $3$.

\begin{Proposition} Let $X$ be a projective curve, $L$ be a line bundle and $V\subseteq H^0(\mathcal{L})$ an embedding subspace. Then, $M_{\mathcal{L},V}^{\vee}\otimes \mathcal{L}^{k+1}$ is $\mathcal{L}^a$-Ulrich if and only if either
     \begin{enumerate}
         \item[i)] $X=\mathbb{P}^1$, $\mathcal{L}=\mathcal{O}_{\mathbb{P}^1}(2)$, $V=H^0(\mathcal{O}_{\mathbb{P}^1}(2))$ and $a=k+2$, thus $X$ is a plane non-singular conic; or
           \item[ii)] $X=\mathbb{P}^1$, $\mathcal{L}=\mathcal{O}_{\mathbb{P}^1}(1)$, $V=H^0(\mathcal{O}_{\mathbb{P}^1}(1))$ and $a=k+3.$
     \end{enumerate} \end{Proposition}

\begin{proof}
We have the following exact sequence
\begin{equation}\label{DualMka}
    0\to \mathcal{L}^{k-a}\to V\otimes \mathcal{L}^{k+1-a}\to M_{\mathcal{L},V}^{\vee}\otimes \mathcal{L}^{k+1-a}\to 0.
\end{equation}
    If $M_{\mathcal{L},V}^{\vee}\otimes \mathcal{L}^{k+1}$ is $\mathcal{L}^a$-Ulrich, then $h^i(M_{\mathcal{L},V}^{\vee}\otimes \mathcal{L}^{k+1-a})=0$ for any $i$. Thus, 
    \begin{equation*}
        \chi(\mathcal{L}^{k-a})=(n+1)\chi(\mathcal{L}^{k+1-a}).
    \end{equation*}
    which is equivalent to: 
    \begin{equation}\label{m}
        m(n(a-k-1)-1)=n(1-g),
    \end{equation}
with $m= \deg L$.
  Note that $g=1$ is impossible, because in that case the left-hand side of (\ref{m}) is zero, thus  $n=1$ and $a=k+2$, but $n=1$ implies $X=\mathbb{P}^1$, which contradicts $g=1$. 
  
 Since, according to Lemma \ref{Lemma for dual M}, the left-hand side of (\ref{m}) is positive, it follows that $g=0$ and 
        \begin{equation*}
            m=\frac{n}{n(a-k-1)-1}.
        \end{equation*}
By Lemma \ref{Lemma for dual M} $a>k+1$, therefore only the following cases are possible:
    \begin{enumerate}
 \item $a=k+2:$ the only solution is $n=2$, $m=2.$ Therefore, $\mathcal{L}=\mathcal{O}_{\mathbb{P}^1}(2)$ and $ V=H^0(\mathcal{O}_{\mathbb{P}^1}(2))$. From:

 $$0\to\mathcal{O}_{\mathbb{P}^1}(-2)\to \mathbb{C}^3\otimes \mathcal{O}_{\mathbb{P}^1}\to M_{\mathcal{L},V}^{\vee} \to 0, $$
 it follows that
 $M_{\mathcal{L},V}^{\vee}=\mathcal{O}_{\mathbb{P}^1}(1)\oplus \mathcal{O}_{\mathbb{P}^1}(1)$. It is readily checked that under these conditions $M_{\mathcal{L},V}\otimes \mathcal{L}^{k+1}$ is $\mathcal{L}^{k+2}-$Ulrich. 
 \item $a=k+3:$ the only solution is $n=1$, $m=1$. Therefore, $\mathcal{L}=\mathcal{O}_{\mathbb{P}^1}(1)$, $ V=H^0(\mathcal{O}_{\mathbb{P}^1}(1))$, and $X=\mathbb{P}^1$. In this case $M_{\mathcal{L},V}^{\vee}=\mathcal{O}_{\mathbb{P}^1}(1)$ and $M_{\mathcal{L},V}^{\vee}\otimes \mathcal{L}^{k+1}$ is $\mathcal{L}^{k+3}-$Ulrich. 
\item $a>k+3:$ this case is impossible, because we obtain: 
$$m <\frac{n}{2n-1}\le 1.$$ 
\end{enumerate}
\end{proof}

Now, we address the case of surfaces:

\begin{Proposition}\label{dual M for surfaces}
  Let $X$ be a surface, $\mathcal{L}$ a very ample line bundle and $V\subseteq H^0(\mathcal{L})$ an embedding subspace, $k\in \mathbb{Z}$, and $a\in \mathbb{N}$. Then, $M^{\vee}_{\mathcal{L}, V}\otimes \mathcal{L}^{k+1}$ is $\mathcal{L}^a-$Ulrich if and only $k=0$, $a=2$ and $X=\mathbb{P}^2$, that is for the sheaf $T\mathbb{P}^2$ with respect to $\mathcal{O}_{\mathbb{P}^2}(2)$ (compare with \cite{Montero}).
 \end{Proposition}

 \begin{proof}
     If $M^{\vee}_{\mathcal{L}, V}\otimes \mathcal{L}^{k+1}$ is $\mathcal{L}^a-$Ulrich, then we have:

     \begin{align*}
         \chi(\mathcal{L}^{k-a}) &=(n+1)\chi(\mathcal{L}^{k+1-a}) \\
         \chi(\mathcal{L}^{k-2a}) &=(n+1)\chi(\mathcal{L}^{k+1-2a}).
     \end{align*}
Subtracting these two equations and using Riemann-Roch, we obtain:
$$ g_H-1 = -a^2 L^2 + (k+1-a)L^2 + \frac{aL^2}{n},$$
with $H=aL$. By Lemma \ref{Lemma for dual M}, we have $k+1<a$. Thus, the right-hand side of the equality is negative, and we conclude that $g_H=0$.

Therefore,

$$1= (a-(k+1)-\frac{1}{n})aL^2.$$

Taking into account that $2\le n$, $0\le k+1<a$, and $L^2=1$ if and only if $n=2$, or equivalently $X=\mathbb{P}^2$, it is easy to deduce that the only possible solution is $k=0$, $a=2$, and $n=2$. 

By the isomorphism (\ref{tangent isomorphisms}), we have that $M_{\mathcal{L},V}^{\vee}\otimes \mathcal{L}\simeq T\mathbb{P}^2$. Using the Euler exact sequence:

$$0\to \mathcal{O}_{\mathbb{P}^2}\to \mathcal{O}_{\mathbb{P}^2}(1)^{\oplus 3} \to T\mathbb{P}^2 \to 0   $$
we readily check that, in fact, $T\mathbb{P}^2$ is $\mathcal{O}_{\mathbb{P}^2}(2)-$Ulrich.
 \end{proof}

 Now, we are in a position to prove:

 \begin{Proposition}
  \label{dual M for varieties}
  Let $X$ be a projective variety of dimension at least $3$, $\mathcal{L}$ a very ample line bundle, and $V\subseteq H^0(\mathcal{L})$ an embedding subspace. Then, for any $k\in \mathbb{Z}$, and $a\in \mathbb{N}$, $M^{\vee}_{\mathcal{L}, V}\otimes \mathcal{L}^{k+1}$ is not $\mathcal{L}^a-$Ulrich.    
 \end{Proposition}

 \begin{proof}  

 The case $a=1$ was treated in \cite{Lopez-Collino}. We give here another argument, for the sake of completeness. By Lemma \ref{Lemma for dual M}, if $a=1$, the only possibility is $k=-1$. The exact sequence (\ref{DualMkvariedades}) becomes:

 $$0\to \mathcal{L}^{-1}\to V\otimes \mathcal{O}_X \to M_{\mathcal{L},V}^\vee \to 0.$$

 Thus, for Kodaira Vanishing $h^0(M^\vee)=n+1$
 on the other hand, if $M^\vee$ is $\mathcal{L}-$Ulrich, then $h^0(M^\vee)= \deg X .\operatorname{rank} M^\vee= \deg X. n$. This gives the contradiction $\deg X= \frac{n+1}{n}$.

Thus, we assume $a>1$ and proceed by induction on $d=\dim X$. Recall the isomorphisms (\ref{tangent isomorphisms}). The hypothesis of induction means that given $X\hookrightarrow \mathbb{P}^n$ with $\dim X=d\ge 3$, we assume that for any projective variety

$$Z\hookrightarrow \mathbb{P}^m $$
with $\dim Z= d-1$, the only possibility for 

$$T\mathbb{P}^m(k)\vert_Z \simeq M^\vee_{
\mathcal{O}_Z(1),\mathbb{P}^m} \otimes \mathcal{O}_Z(k +1)$$
being $\mathcal{O}_Z(a)-$Ulrich is $a = 2$, $k = 0$ and $m = d-1$, i.e. $Z = \mathbb{P}^{d-1}$.

In fact, we use the hypothesis of induction only in the following case:
let $F\subset \mathbb{P}^n$ be a hypersurface of degree $a$ such that $Y:=X\cap F$ is irreducible and non-singular (such $F$ exists by Bertini´s Theorem). 
If $T\mathbb{P}^n(k)\vert_X$ is Ulrich, then by Proposition \ref{proposition properties of Ulrich}ii), $T\mathbb{P}^n(k)\vert_Y$ also is. We claim that, since $a>1$, $Y\subseteq \mathbb{P}^n$ is non degenerated. In fact, if $Y=F\cap X=:F_X$ and $Y\subset H_X^{\prime}$ for some hyperplane $H^{\prime}$, then as divisors $F_X\le H^{\prime}_X$. Thus, we obtain $aH_X\sim F_X \le H_X$, that contradicts $a>1$.

Therefore:

$$(T\mathbb{P}^n(k)\vert_X)\vert_Y = T\mathbb{P}^n(k)\vert_Y;$$

but then, by induction, $Y=\mathbb{P}^{d-1}$, which is impossible. 
\end{proof}

    \section{Proof of Theorem \ref{Theorem Main} ii)}\label{Item ii)}

In this section, we conclude the proof of Theorem \ref{Theorem Main}. We proceed in broad outline as in the previous section, establishing the result first for curves, then for surfaces, and finally for higher-dimensional varieties.

The starting point is:

\begin{Lemma}\label{lemma M} Let $\mathcal{L}$ be a very ample line bundle on the projective variety $X$ and $V\subseteq H^0(\mathcal{L})$ an embedding subspace. Then, $h^0(M_{\mathcal{L},V}\otimes \mathcal{L}^k) \ne 0$ if and only if $k\ge 1$. In particular, if $M_{\mathcal{L},V}\otimes \mathcal{L}^{k-1}$ is $\mathcal{L}^a-$Ulrich, then:
$$ 0<k-1 \le a.$$
\end{Lemma}

\begin{proof} The proof of this Lemma is rather technical; to complete the details, the reader can consult \cite{Brinzanescu}, Chapter 1, \cite{Friedman}, and \cite{Kobayashi}, Chapter 5. If $k\le 0$, then $h^0(M_{\mathcal{L}, V}\otimes \mathcal{L}^k)=0$, because $h^0(M_{\mathcal{L}, V})=0$. Indeed, the exact sequence:

\begin{equation}\label{M} 0\to M_{\mathcal{L}, V}\to V\otimes \mathcal{O}_X \to \mathcal{L} \to 0,\end{equation}
induces an inclusion $H^0(V)\hookrightarrow H^0(\mathcal{L}).$

Thus, to prove the Lemma, it is sufficient to see that $h^0(M_{\mathcal{L}, V}\otimes \mathcal{L})\ne 0$. From the dual sequence of (\ref{M}), we see that $h^0(M_{\mathcal{L}, V}^\vee) \ge n+1$ and $rank (M_{\mathcal{L}, V}^\vee)=n$.

Using $n-1$ sections of $M_{\mathcal{L}, V}^\vee$, we can  construct an exact sequence:

    $$0 \to  \mathcal{O}_X^{n-1} \overset {ev}{\to} M_{\mathcal{L}, V}^\vee \to \widetilde{\mathcal{F}}\to 0,$$
   where $ev$ is the evaluation map, and the set of singular points of $\widetilde{\mathcal{F}}$ has codimension two. Moreover, $\det \widetilde{\mathcal{F}}= \det M_{\mathcal{L}, V}^\vee= \mathcal{L}$.
   
    Let  $\tau (\widetilde{F})$ be the torsion subsheaf of $\widetilde{\mathcal{F}}$. Since $\operatorname {Supp} \tau(\widetilde{F})\subset  \text{Sing}(\widetilde{\mathcal{F}})$, it follows that $\operatorname {Supp} \tau (\widetilde{F})$ has codimension at least two  and, therefore,  $\text{det}(\tau(\widetilde{F}))=\mathcal{O}_X$. We have an exact sequence:

    $$0\to \tau(\widetilde{\mathcal{F}})\to \widetilde{\mathcal{F}} \to \mathcal{F}\to 0,$$
with $\mathcal{F}$ torsion free and $\det \mathcal{F}= \det \widetilde{\mathcal{F}}=\mathcal{L}$.
   
Thus a diagram

$$\xymatrix{ 0\ar[r]& \mathcal{O}_X^{n-1} \ar[r] \ar[d]& M_{\mathcal{L}, V}^\vee \ar[r] \ar[d]^{id} &\widetilde{\mathcal{F}} \ar[r] \ar[d] & 0 \\ 0 \ar[r] & \mathcal{G} \ar[r] & M_{\mathcal{L}, V}^\vee \ar[r] & \mathcal{F} \ar[r]\ar[d] & 0\\ & & & 0}
$$
exists. From this, we obtain the exact sequence:

\begin{equation}\label{sequence F}0\to \mathcal{F}^\vee \to M_{\mathcal{L}, V} \to \mathcal{G}^{\vee}.\end{equation}
Now, $\mathcal{F}^\vee$ is reflexive and of rank one, thus invertible; and, since $\mathcal{F}$ is torsion free:

$$\det \mathcal{F}=\mathcal{L}= (\mathcal{F}^\vee)^\vee.$$

Thus, $\mathcal{L}= (\mathcal{F}^\vee)^{-1}$, or $\mathcal{F}^\vee =\mathcal{L}^{-1}$. It follows from the exact sequence (\ref{sequence F}) that $M_{\mathcal{L}, V}\otimes \mathcal{L}$ has a section.

The last statement follows from the fact that if $M_{\mathcal{L},V}\otimes \mathcal{L}^{k-1}$ is $\mathcal{L}^a-$Ulrich, then $h^0(M_{\mathcal{L},V}\otimes \mathcal{L}^{k-1})\ne 0$ and $h^0(M_{\mathcal{L},V}\otimes \mathcal{L}^{k-1-a})=0$.
\end{proof}

\begin{Proposition}\label{Proposition M for curves}
    Let $X$ be a projective curve of genus $g$, $L$ an ample line bundle, and $V\subseteq H^0(\mathcal{L})$ an embedding subspace. Then $M_{\mathcal{L},V}\otimes \mathcal{L}^{k-1}$ is $\mathcal{L}^a$-Ulrich if and only if $a=k-1$ and  $X$ is a normal rational curve.
    \end{Proposition}
\begin{proof}  We have the exact sequence 
\begin{equation}\label{Mka}
0\to M_{\mathcal{L}, V} \otimes \mathcal{L}^{k-1-a}\to V\otimes \mathcal{L}^{k-1-a}\to \mathcal{L}^{k-a}\to 0
\end{equation}
    We consider the following three cases:
    \begin{enumerate}
        \item[(a)] $k-1=a:$  From the exact sequence (\ref{Mka}), $M_{\mathcal{L}, V}\otimes \mathcal{L}^{k-1}$ is $\mathcal{L}^a$-Ulrich if and only if it satisfies the following properties: 
      \begin{itemize}
          \item $V=H^0(\mathcal{L})$
          \item $h^1(\mathcal{L})=(n+1)g.$
      \end{itemize}
Therefore, we obtain from Riemann-Roch that  $X=\mathbb{P}^1$, $\mathcal{L}=\mathcal{O}_{\mathbb{P}^1}(n)$ and $V=H^0(X,\mathcal{L}).$ In this case $M_{\mathcal{L}, V}=\oplus \mathcal{O}_{\mathbb{P}^1}(-1)$. Therefore, $M_{\mathcal{L}, V}\otimes \mathcal{L}^{k-1}=\oplus \mathcal{O}_{\mathbb{P}^1}(a-1),$ which is $\mathcal{L}^a$-Ulrich.

 \item[(b)] $k=a$: From the exact sequence (\ref{Mka}), we get $H^1(M_{\mathcal{L}, V}\otimes \mathcal{L}^{k-1-a})\neq 0.$ Therefore, $M_{\mathcal{L}, V}\otimes \mathcal{L}^{k-1}$ is not $\mathcal{L}^a$-Ulrich.

        \item[(c)] $k<a:$ From the exact sequence (\ref{Mka}), $M_{\mathcal{L}, V}\otimes \mathcal{L}^{k-1}$ is $\mathcal{L}^a$-Ulrich if and only if
        \begin{equation*}
            (n+1)\chi(\mathcal{L}^{k-1-a})=\chi(\mathcal{L}^{k-a}),
        \end{equation*}
        which, by Riemann–Roch, is equivalent to
        \begin{equation*}
            m\left(n(a+1-k)+1\right)=n(1-g),
        \end{equation*}
        were $m=\deg L$.
        Since, by hypothesis, the left-hand side is positive, it follows that $X=\mathbb{P}^1$ and 
        \begin{equation*}
            m=\frac{n}{n(a+1-k)+1}.
        \end{equation*}
        which is impossible because $m$ is a natural number.
     \end{enumerate}
\end{proof}

Finally, we have:

\begin{Proposition}\label{Proposition M} Let $X\subseteq \mathbb{P}^n$ be a projective variety of dimension $d\ge 2$, $\mathcal{L}$ a very ample line bundle and $V\subseteq H^0(\mathcal{L})$ an embedding subspace. Then, for every $k\in \mathbb{Z}$ and $a\in \mathbb{N}$, the sheaf $M_{\mathcal{L},V}\otimes \mathcal{L}^{k-1}$ is not $\mathcal{L}^a-$Ulrich.
\end{Proposition}
\begin{proof} We have the exact sequence:

\begin{equation}\label{exact cotangent sequence}
    0\to M_{\mathcal{L},V}\otimes \mathcal{L}^{k-1}\to V\otimes \mathcal{L}^{k-1} \to \mathcal{L}^k \to 0.
    \end{equation}
    
The case $a=1$ was considered in \cite{Lopez-Collino}. Once again, we give a different argument here. 

If $a=1$, then by Lemma \ref{lemma M} $k=2$. In this case (\ref{exact cotangent sequence}) gives:

$$0\to M_{\mathcal{L},V}\otimes \mathcal{L} \to V\otimes \mathcal{L} \to \mathcal{L}^2\to 0.$$
Twisting by $\mathcal{L}^{-2}$, we obtain:

$$0\to H^0(\mathcal{O}_X) \to H^1(M_{\mathcal{L},V}\otimes \mathcal{L}\otimes \mathcal{L}^{-2}).$$

Thus, $H^1(M_{\mathcal{L},V}\otimes \mathcal{L}\otimes \mathcal{L}^{-2})\ne 0$ and $M_{\mathcal{L},V}\otimes \mathcal{L}$ cannot be $\mathcal{L}-$Ulrich.

Therefore, assume in the sequel that $a>1$.

Assume first that $\dim X=2$. From the exact sequence (\ref{exact cotangent sequence}) we see that 

$$(n+1)h^2(\mathcal{L}^{k-1-2a})=h^2(\mathcal{L}^{k-2a}),$$
or, by duality,
$$(n+1)h^0(\mathcal{L}^{2a+1-k}(K_X))=h^0(\mathcal{L}^{2a-k}(K_X)).$$

But, since  $k\le 2a$, $h^0(\mathcal{L}^{2a-k}(K_X))\le h^0(\mathcal{L}^{2a+1-k}(K_X))$. Thus,
    $$\vert (2a+1-k)L+K_X\vert = \emptyset .$$

On the other hand $((2a+1-k)L)^2\ge (a+2)^2L^2\ge 9,$
and we are in a position to apply Reider´s Theorem (see \cite{Friedman}, Chapter 9, \cite{Reider}), which in our case assures that $\forall p\in X$, there exists an effective non-zero divisor $E$ containing $p$, such that $E.(2a+1-k)L$ is equal either to $0$ or $1$. The former case is impossible, because $L$ is ample, and the latter because $2a+1-k>1$. Thus, the theorem is valid for surfaces.

If $\dim X\ge 3$, we use, once again, induction on $d=\dim X$. Recall that if $X\subseteq \mathbb{P}(V)$ is non degenerated, then $M_{\mathcal{L},V}\otimes \mathcal{L}^{k-1}\simeq \Omega \mathbb{P}^n(k)\vert_X$. Therefore our induction hypothesis is: if $Z\subseteq \mathbb{P}^m$ is no degenerated and $\dim Z=d-1$, then $\Omega \mathbb{P}^m(k)\vert_Z$ is not $\mathcal{O}_{\mathbb{P}^n}(a)-$Ulrich.
    
Let $Y:=X\cap F$ with $F\in \mathbb{P}H^0(\mathcal{O}_{\mathbb{P}^n}(a))$ general. If $\Omega \mathbb{P}^n(k)\vert_X$ is $\mathcal{O}_X(a)-$Ulrich, then $\Omega \mathbb{P}^n(k)\vert_Y$ is $\mathcal{O}_Y(a)-$Ulrich.
Since $a>1$, $Y\subseteq \mathbb{P}^n$ is non-degenerate (see the proof of \ref{dual M for varieties}), and we obtain a contradiction with the induction hypothesis.
\end{proof}

\begin{Remark}
    The inductive step in the proof of Proposition \ref{Proposition M} can also be justified as follows: if $a>1$, then $H^0(\mathcal{L}(-a))=H^1(\mathcal{L}(-a))=0$, thus $H^0(\mathcal{L})\simeq H^0(\mathcal{L}\vert_Y)$. Looking at the diagram:

    $$\xymatrix{ 0\ar[r]& M_{\mathcal{L},V} \ar[r] \ar[d]& V\otimes \mathcal{O}_X \ar[r] \ar[d]  &\mathcal{L} \ar[r] \ar[d] & 0 \\ 0 \ar[r] & M_{\mathcal{L}\vert_Y, V} \ar[r]\ar[d] & V\otimes \mathcal{O}_Y \ar[r] \ar[d] & \mathcal{L}\vert_Y \ar[r]\ar[d] & 0\\ &  0 & 0 & 0} $$
    with vertical arrows determined by restrictions, one readily deduced that $(M_{\mathcal{L},V})\vert_Y \simeq M_{\mathcal{L}\vert_Y, V}$. Thus, induction can be performed directly using the syzygy bundles. Taking duals a similar argument works for the proof of Proposition \ref{dual M for varieties}.
\end{Remark}

    \section{Ulrichness with respect to an arbitrary polarization}\label{section arbitrary H}

    Hitherto, we have studied Ulrichness of twisted syzygies bundles with respect to $\mathcal{L}^a$. It is natural to ask about Ulrichness with respect to any polarization. Some previously proven results can be generalized to the case of $H$-Ulrichnesss, for $H$ any very ample divisor. For instance:

    \begin{Lemma}\label{Lemma emptiness} Let $X$ be a projective variety of dimension at least $2$, $L$, $H$ very ample divisors and $V\subseteq H^0(L)$ an embedding subspace. If $M^{\vee}_{\mathcal{L},V}\otimes \mathcal{L}^{k+1}$ is $H-$Ulrich, then $k\ge -1$ and:

$$\vert (k+1)L-H\vert =\emptyset .$$ 
\end{Lemma}
\begin{proof}
For any $k\in \mathbb{Z},$ we have the exact sequence
\begin{equation}\label{DualMkvariedades}
    0\to \mathcal{L}^{k}\to V\otimes \mathcal{L}^{k+1}\to M_{\mathcal{L}, V}^{\vee}\otimes \mathcal{L}^{k+1}\to 0.
\end{equation}

If $k\leq -2,$ by Kodaira Vanishing Theorem and from the exact sequence (\ref{DualMkvariedades}), we have $h^0(M_{\mathcal{L}, V}^{\vee}\otimes \mathcal{L}^{k+1})=0$. Therefore, if $k\leq -2$ then $M_{\mathcal{L}, V}^{\vee}\otimes \mathcal{L}^{k+1}$ is not $H$-Ulrich.
    
    If $k\ge 0$, from the hypotheses on Ulrichness, we have:

    $$h^0(\mathcal{L}^{k+1}(-H))\ge (h^0(\mathcal{L}^k(-H))=(n+1)h^0(\mathcal{L}^{k+1}(-H)).$$

Thus, we must have $h^0(\mathcal{L}^{k+1}(-H))=0$ .  
\end{proof}

Now, we restrict ourselves to the case of surfaces. We keep the previous notation for the remainder of this section: $\mathcal{L}$ will be a very ample line bundle, $V\subseteq H^0(\mathcal{L})$ an embedding subspace of dimension $n+1$. A similar argument to the one used in the proof of Proposition \ref{dual M for surfaces} can be developed.

\begin{Proposition}\label{Proposition conditions on L and H} With the previous notation, assume that $X$ is a projective surface and $H$ is a very ample divisor.

\begin{itemize} \item[i)] if $M_{\mathcal{L}, V}^{\vee}\otimes \mathcal{L}^{k+1}$ is $H-$Ul\-rich, then 

$$H.K_X= 2(k+1)L.H-3H^2 + 2\frac{L.H}{n},$$ or equivalently:

$$g_H-1= (k+1)L.H-H^2 + \frac{L.H}{n}.$$ 

\item[ii)] if $M_{\mathcal{L}, V}\otimes \mathcal{L}^{k-1}$ is $H-$Ulrich, then 

$$H.K_X=2(k-1)L.H -3H^2 -2 \frac{L.H}{n},$$
or equivalently

$$g_H-1= (k-1)L.H -H^2 -\frac{L.H}{n}.$$
\end{itemize}
\end{Proposition}
\begin{proof} i) Being $M_{\mathcal{L}, V}^{\vee}\otimes \mathcal{L}^{k+1}$ Ulrich, we have:

$$H^i(M_{\mathcal{L}, V}^{\vee}\otimes \mathcal{L}^{k+1}(-H))=H^i(M_{\mathcal{L}, V}^{\vee}\otimes \mathcal{L}^{k+1}(-2H))=0$$
for every $i\ge 0$.

Thus, from the exact sequence (\ref{DualMka}), we obtain:

$$h^i(\mathcal{L}^k(-H))=(n+1)h^i(\mathcal{L}^{k+1}(-H))$$

and

$$h^i(\mathcal{L}^k(-2H))=(n+1)h^i(\mathcal{L}^{k+1}(-2H))$$
for $i=0,1,2$.

Combining the previous relations for $H$ and $2H$ we obtain:

$$\chi(\mathcal{L}^k(-H))-\chi(\mathcal{L}^{k}(-2H))= (n+1)(\chi(\mathcal{L}^{k+1}(-H))-\chi(\mathcal{L}^{k+1}(-2H))).$$

The result follows by applying the Riemann-Roch theorem for surfaces.

ii) Quite analogous to part i).
\end{proof}

\begin{Corollary}\label{Corollary simultaneous} For any $k\in \mathbb{Z} $,  $M_{\mathcal{L}, V}^{\vee}\otimes \mathcal{L}^{k+1}$ and $M_{\mathcal{L}, V}\otimes \mathcal{L}^{k-1}$ can not be simultaneously $H-$Ulrich. \end{Corollary} 

\begin{proof} Otherwise, by Proposition \ref{Proposition conditions on L and H} we must have:

$$(k+1)L.H-H^2+\frac{L.H}{n}= (k-1)L.H-H^2-\frac{L.H}{n},$$
or
$$2(1+\frac{1}{n})L.H=0.$$
\end{proof}

We now recall the concept of the dual Ulrich. Given $\mathcal{E}$ a locally free sheaf on a projective variety $X$ of dimension $d$, its dual Ulrich with respect to a very ample divisor $H$ is the sheaf $\mathcal{E}^\vee((d+1)H+K_X)$. If $\mathcal{E}$ is $H-$Ulrich, then its Ulrich dual with respect to $H$  also is. 

This can be used to construct examples of twisted syzygies bundles that are not Ulrich with respect to different polarizations. 

\begin{Corollary}\label{corollary 2kL}
    If $\mathcal{L}$ satisfies:

    $$\mathcal{L}^{2k}=\mathcal{O}_X(3H+K_X),$$
    then $M_{\mathcal{L}, V}^\vee\otimes \mathcal{L}^{k+1}$ is not $H-$Ulrich.
\end{Corollary}

\begin{proof}
    Under these conditions, the Ulrich dual of $M_{\mathcal{L}, V}^\vee\otimes \mathcal{L}^{k+1}$ with respect to $H$ coincides with $M_{\mathcal{L}, V}\otimes \mathcal{L}^{k-1}$ and the claim follows from Corollary \ref{Corollary simultaneous}
\end{proof}

The following examples deal with the case $X=\mathbb{P}^1\times \mathbb{P}^1$. If $C_1$, $C_2$ are generators of the two rulings, we write $\mathcal{O}(a,b)$ for the line bundle associated with the divisor $aC_1+bC_2$.

\begin{Example}
    If $\mathcal{L}=\mathcal{O}(1,4)$, $H$ the divisor of class $(2,6)$, then $M_{\mathcal{L}, V}^{\vee}\otimes \mathcal{L}^3$ is not $H-$Ulrich. It follows from Corollary \ref{corollary 2kL}.
\end{Example}

\begin{Example} Consider $\mathcal{L}= \mathcal{O}(1,1)$ and let $H$ be a very ample divisor, say of class $(a,b)$, $a,b>0$. Let $V=H^0(\mathcal{L})$. Then, for any $H$, $\mathcal{E}:= M_{\mathcal{L}, V}\otimes \mathcal{L}^{k-1}$ is not $H-$Ulrich.

Indeed, assume $\mathcal{E}$ is $H-$Ulrich. Using that $\mathcal{L}^2=2$, $\mathcal{L}.K_X=-4$, $\mathcal{L}.H=a+b$, $H.K_X=-2(a+b)$, and $H^2=2ab$, the conditions on Chern´s classes given in Proposition 5.1.1 of \cite{CostaMirog} become:

\begin{equation}\label{c_1}
   9ab =(3k-1)(a+b), \end{equation}

   \begin{equation}\label{c_2} 6ab =(3k+1)(k-1).
\end{equation}

In addition, we have the condition:

$$\chi (\mathcal{L}^k)=(n+1)\chi(\mathcal{L}^{k-1})= 4\chi(\mathcal{L}^{k-1});$$ that amount to:

$$(3k-1)(a+b)=3ab +3k^2-2k-4,$$
which is incompatible with the relations (\ref{c_1}) and (\ref{c_2}).

\end{Example}

\end{document}